\documentclass[12pt]{amsart}
\usepackage{graphicx}
\usepackage{amssymb}
\usepackage{stmaryrd}
\usepackage{enumitem}

\usepackage{geometry}
\usepackage{tikz}
\usepackage{float}
\usepackage{url}
\usepackage{mathtools}
\usepackage{amsmath,amsfonts,amsthm}
\usepackage{mathrsfs} 
\newcommand{\irr}{\operatorname{irr}}
\newcommand{\case}[1]{\paragraph*{Case #1:}}
\newtheorem{hypothesize}{Hypothesize}
\usepackage{mathtools}

\DeclarePairedDelimiter\floor{\lfloor}{\rfloor}
\usepackage{amsmath}
\usepackage{systeme}

\newtheorem{theorem}{Theorem}[section]
\newtheorem{lemma}[theorem]{Lemma}
\newtheorem{proposition}{Proposition}[section]
\newtheorem{corollary}[theorem]{Corollary}
\newtheorem{definition}{Definition}
\newtheorem{example}{Example}

\author{Jasem Hamoud}
\address{\textbf{Jasem Hamoud} 
Department of Discrete Mathematics, Moscow Institute of Physics and Technology
}
\email{khamud@phystech.edu}
\thanks{}

\author{Duaa Abdullah}
\address{\textbf{Duaa Abdullah:} Department of Discrete Mathematics, Moscow Institute of Physics and Technology }
\email{abdulla.d@phystech.edu}
\thanks{}

\title{Topological Indices Among Strong Support Vertex}

\date{}
\begin{document}

\begin{abstract}
In this paper, we provide the irregularity properties of trees with strong support vertex by analyzing two prominent topological indices: the Albertson index and the Sigma index. We further establish extremal bounds for both indices across families of trees defined by given degree sequences. Let $T_1$ and $T_2$ be a star trees of order $n$, where $T_1 \cong T_2$, then, we provide Albertson index of $T_1 \cong T_2$. Let $\mathcal{T}_{n, \Delta}$ be a class of trees with $n$ vertices, there are a tree $T^{\prime} \in \mathcal{T}_{n, \Delta}$ such that $\irr(T^{\prime}) < \irr(T)$.

\end{abstract}

\maketitle

\noindent\rule{12.7cm}{1.0pt}


\medskip

\noindent

\medskip


\noindent\rule{12.7cm}{1.0pt}

\section{Introduction}\label{sec1}
Throughout this paper. Let $G=(V,E)$ be a simple graph of order $n$, with vertices set $V=\{v_1,v_2,\dots,v_n\}$ and edges set $E=\{e_1,e_2,\dots,e_m\}$. Let $e=u v \in E(G)$ an edge, then the imbalance of $e$ known as $\left|d_C(u)-d_G(v)\right|$ it is employee in ``Albertson Index''.  Dorjsembe,S., et al. in~\cite{Dorjsembe2023GutmanLI} for a path $u_0 u_1 \cdots u_t$ in graph $G$, then $imb_G\left(u_0, u_t\right)=\sum_{i=0}^{t-1}\left|d_G\left(u_i\right)-d_G\left(u_{i+1}\right)\right|$. In 1997, Albertson in~\cite{ALBERTSON} mention to the imbalance of an edge $uv$ by $imb(uv)$ where we considered a graph $G$ is regular if all of its vertices have the exact same degree, then the irregularity measure defined in~\cite{ALBERTSON, GUTMAN2, abdo2014total} as: 
\[
\operatorname{irr}(G)=\sum_{uv\in E(G)}\lvert d_u(G)-d_v(G) \rvert.
\]
 Ali, A. et al in~\cite{Albalahi2025DimitrovAli} mention to total irregularity of $\mathscr{D}(G)=(d_1,d_2,\dots,d_i)$ a degree sequence of $G$ where $d_1\ge d_2\ge\cdots\ge d_n$, it defined in Definition~\ref{degreeseq}, then we have $\operatorname{irr}(G)=2(n+1)m - 2 \sum_{i=1}^n id_i$. 
Thus, a``total irregularity of Albertson index'', defined in \cite{IrregularityGutmanKulli,AbdoGutmanDimitrov, Ali2023Abeer,Buyantogtokh2021}  as:
\[
\operatorname{irr}_T(G)=\sum_{|u, v| \leq V(G)}\left|d_G(u)-d_G(v)\right|, \quad \operatorname{irr}_T(G)=\sum_{\{u,v\}\in E(G)}^{} |d_G(v)-d_G(u)|.
\]
Andrade, E., et al. in~\cite{AndradeRobbianoLenes} introduce a topological index of graph $G$ with $m \geq 1$ edges, it known ``The Zagreb index'' of $G$, denote by $Z_{g}(G)$, then $Z_{g}(G)=\sum_{i \in V(G)} d^{2}(i)$, for $i=1,2, \ldots, k$, where $k=\min \{n, m\}$.The first and the second Zagreb index, $M_1(G)$ and $M_2(G)$ are defined in~\cite{TRINAJSTIC, WILCOX,ALBERTSON} as: 
\[
M_1(G)=\sum_{i=1}^{n}d_i^2, \quad \text{and} \quad M_2(G)=\sum_{uv\in E(G)} d_u(G)d_v(G).
\] 
Actually, an alternative expressions introduced by M. Matej\'i et al. in~\cite{ALBERTSON, Nikoli} for the first Zagreb index  as $M_1(G)=\sum_{u\sim v}(d_u+d_v)$.  The recently introduced $\sigma(G)$ irregularity index is a simple diversification of the previously established Albertson irregularity index, in~\cite{AbdoGutmanDimitrov,Mandal2022Prvanovic,YurttasGutmanTogan,Yang2023Deng,Albalahi2025DimitrovAli,Ali2023Abeer} defined as: 
\[
\sigma(G)=\sum_{uv\in E(G)}\left( d_u(G)-d_v(G) \right)^2.
\]
Also in \cite{Mandal2022Prvanovic,Albalahi2025DimitrovAli} mention to relation between Albertson index and Sigma index as $\sqrt{\sigma(G)} \leq \rm{irr}(G) \leq \sqrt{m\sigma(G)}$, where $\rm{irr}(G) \leq (n-1)(n-2)$, for any tree of order $n\geq 2$ with equalities if and only if $G$ is regular.\par 
This paper is organized as follows. Section~\ref{sec2} reviews relevant definitions and preliminary results. In section~\ref{sec3}, we present our main theorems concerning the Albertson index in trees, by subsection~\ref{subsec1}, we provide Albertson Index among strong support vertex. In section~\ref{sec4}, we provide advanced results for Sigma index in trees among strong support vertex.
\section{Preliminaries}\label{sec2}
In this section, we presented several fundamental concepts that will be deployed for the major results in the section~\ref{sec3} where we considered $\mathscr{D}(G)=\left(d_{G}\left(v_{1}\right), d_{G}\left(v_{2}\right), \ldots, d_{G}\left(v_{n}\right)\right)$ be a degree sequence of $G$ defined in Definition~\ref{degreeseq} had different with Asymptotic Degree Sequence by Molloy and Reed in~\cite{Molloy95Reed}. In Definition~\ref{TreeIsomorphic}, we presented Tree Isomorphic for employee in Theorem~\ref{resn1}. 

\begin{definition}[Degree Sequence~\cite{AshrafiGhalavand,Zhang2013Gray}]~\label{degreeseq}
Let $G=(V,E)$ be a simple graph, where $V=\{v_{1}, v_{2}, \ldots, v_{n}\}$, let $\mathscr{D}(G)=\left(d_{G}\left(v_{1}\right), d_{G}\left(v_{2}\right), \ldots, d_{G}\left(v_{n}\right)\right)$ be a degree sequence of $G$ where $d_{G}\left(v_{1}\right) \geqslant d_{G}\left(v_{2}\right) \geqslant \cdots \geqslant d_{G}\left(v_{n}\right)$. When $\mathscr{D}(G)=(k, k, \ldots, k)$, then $G$ is regular of degree $k$. Otherwise, the graph is irregular
\end{definition}
 Ghalavand, A., Ashrafi, A., R. in~\cite{Ghalavand2020Ashrafi} mention to a graphic sequence of a degree sequence $\mathscr{D}=\left(d_1, d_2, \ldots, d_n\right)$ for a simple graph $G$ with vertex set $V(G)=\left\{v_1, v_2, \ldots, v_n\right\}$ such that $d_i=\operatorname{deg}_G\left(v_i\right)$ where $ 1 \leq i \leq n$. we denote to the degree sequence by $D(G)$, it defined as
\[
D(G)=\left(d_1, d_2, \ldots, d_n\right)=(\underbrace{x_1, \ldots, x_1}_{n_1 \text { times }}, \underbrace{x_2, \ldots, x_2}_{n_2 \text { times }}, \ldots, \underbrace{x_t, \ldots, x_t}_{n_i \text { times }}),
\]
where$n_1,n_2,\dots , n_i$ are positive integers, and $d_1=x_1>x_2>\dots>x_i=d_n$. In case $n_1+n_2+\dots+n_i=k$, then the degree sequence become $D(G)=\left(d_1, d_2, \ldots, d_n\right)=(x_{1}^{n_1},x_{2}^{n_2},\dots,x_{i}^{n_i})$. 
\begin{definition}[Asymptotic Degree Sequence~\cite{Molloy95Reed}]
Let $\mathscr{D}(G)=\left(d_{G}\left(v_{1}\right), d_{G}\left(v_{2}\right), \ldots, d_{G}\left(v_{n}\right)\right)$ be a degree sequence defined in Definition~\ref{degreeseq}, then if $\mathcal{D} = (d_i(v_1), d_i(v_2), \ldots,d_i(v_n))$ it known an asymptotic degree sequence when $d_i(v_1)=d_i(v_2)=\ldots=d_i(v_n)=d_i(v)$, then:
\[
\begin{cases}
    d_i(v) = 0 & \text{ for } i \geq v, \\
    \sum_{i=0}^{n} d_i(v) = n & \text{ for } v \geq i.
\end{cases}
\]
\end{definition}
Let $G$ be a regular graph of order $n$ and size $m$, with a degree sequence $\mathscr{D}(G)=$ $\left(d_{1}, d_{2}, \ldots, d_{n}\right)$ and  let $r>0$ integer number. In~\cite{AshrafiGhalavand} show
\[
\frac{\left(d_{1}+r\right)+\left(d_{2}+r\right)+\cdots+\left(d_{n}+r\right)}{n} \geqslant \sqrt[n]{\left(d_{1}+r\right)\left(d_{2}+r\right) \cdots\left(d_{n}+r\right)}
\]
with determine condition $d_{1}=d_{2}=\cdots=d_{n}$, when $\sum_{i=1}^{n} d_{i}=2 m$ we have
\[
\frac{n^{n}\left(d_{1}+r\right)\left(d_{2}+r\right) \cdots\left(d_{n}+r\right)}{(2 m+r n)^{n}} \leqslant 1.
\]
\begin{definition}[Tree Isomorphic]~\label{TreeIsomorphic}
Let $T_1$ and $T_2$ be a trees of order $n$ with degree sequences $\mathscr{D} = (d_1, d_2, \ldots, d_n)$ and $\mathscr{X}= (x_1, x_2, \ldots, x_n)$, where $d_n \geq \dots \geq d_1$ and $x_n \geq \dots \geq x_1$. We say that $T_1 \cong T_2$ (see Figure~\ref{fig:Isomorphic}) if there exists a bijection
\[
f: V(T_1) \to V(T_2)
\]
such that for all $u, v \in V(T_1)$, $\{u, v\} \in E(T_1)$ where $\{f(u), f(v)\} \in E(T_2)$.
In this case, the degree sequences $\mathscr{D}$ and $\mathscr{X}$ are identical up to the ordering of vertices: $\deg_T(\mathscr{D})=\deg_T(\mathscr{X})$.
\end{definition}
\begin{figure}[H]
    \centering
   \begin{tikzpicture}[scale=.8]
  \node[circle, draw, fill=black, inner sep=2pt, label=below:$v_1$] (v1) at (0,0) {};
  \node[circle, draw, fill=black, inner sep=2pt, label=above:$v_2$] (v2) at (-1,1) {};
  \node[circle, draw, fill=black, inner sep=2pt, label=above:$v_n$] (v4) at (1,1) {};
  \draw (v1) -- (v2);
  \draw (v1) -- (v4);
  \node at (0,-1) {$T_1$};
  \node at (0,1) {$\dots$};
  \node[circle, draw, fill=black, inner sep=2pt, label=below:$u_1$] (u1) at (4,0) {};
  \node[circle, draw, fill=black, inner sep=2pt, label=above:$u_2$] (u2) at (3,1) {};
  \node[circle, draw, fill=black, inner sep=2pt, label=above:$u_n$] (u4) at (5,1) {};
  \draw (u1) -- (u2);
  \draw (u1) -- (u4);
  \node at (4,-1) {$T_2$};
   \node at (4,1) {$\dots$};
\end{tikzpicture}
    \caption{demonstrates $T_1 \cong T_2$.}
    \label{fig:Isomorphic}
\end{figure}
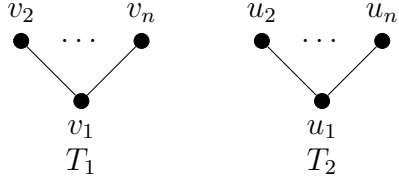
\begin{proposition}\label{three.c}
 Let $\mathscr{D}=(d_1,d_2,d_3)$ a degree sequence where $d_1\geqslant d_2 \geqslant d_3$, then Albertson index define as:
 	\[
 	\irr(T)=\left\lbrace
 	\begin{aligned}
 		& \irr_{max}(T)=(d_1-1)^2+(d_2-1)^2 +(d_3-1)(d_3-2)(d_1-d_3)(d_2-d_3) \\
 		& \irr_{min}(T)=(d_1-1)^2+(d_3-1)^2+(d_2-1)(d_2-2)+(d_1-d_3).
 	\end{aligned} 
 	\right. 
 	\]
 \end{proposition}
 \begin{definition}
Let $\mathcal{S}$ be a class of graphs, then we have $\operatorname{irr}_{\max},\operatorname{irr}_{\min}$ Albertson index of a graph $G$, where: \begin{gather*} \operatorname{irr}_{\max}(\mathcal{S}) =\max \{\operatorname{irr}(G)\mid G\in \mathcal{S}\}, \\
\operatorname{irr}_{\min}(\mathcal{S}) =\min \{\operatorname{irr}(G)\mid G\in \mathcal{S}\}.
\end{gather*}
 \end{definition}

\begin{lemma}[Yang J, Deng H., M.,~\cite{Yang2023Deng}]\label{le2.1}
	Let $G$ has the maximal Sigma index among all connected graphs with $n$ vertices and $p$ pendant vertices, where $n, p$ are positive integers such that $1\leq p \leq n-3$. Then: $\Delta(G)=n-1$. 
\end{lemma}
\section{Main Result}\label{sec3}
In this section, we provide strong support vertex in Definition~\ref{Strongsupportvertex} applying for Proposition~\ref{caterpillarvertex}. By considered $\mathscr{D}(G)=\left(d_{G}\left(v_{1}\right), d_{G}\left(v_{2}\right), \ldots, d_{G}\left(v_{n}\right)\right)$ be a degree sequence of $G$ defined in Definition~\ref{degreeseq} had satisfying Theorem~\ref{resn1}.

\subsection{Albertson Index Among Strong Support Vertex}~\label{subsec1}
\begin{definition}[Strong support vertex]~\label{Strongsupportvertex}
Let $T=(V,E)$ be a tree of order $n$, a strong support vertex $v\in V$ in $T$ is known a vertex that is adjacent to at least one leaf (a vertex of degree 1) or instead that adjacent to at least two pendant
vertices (see~\cite{Dehgardi2025N}). When $\deg(v)\geqslant3$ is a strong support vertex as we show that in Figure~\ref{strongsupportvertex} where the tree has strong support vertices with degree at least 3 as: $v_0$ has degree (4), $v_{0,3}$ has degree (3) and $v_{0,4}$ has degree (3). All of $v_{0,1}$,$v_{0,2}$ called leafs.
\end{definition}
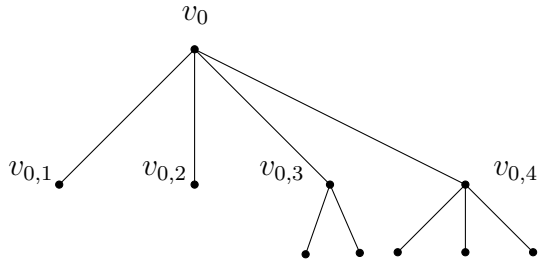
\begin{figure}[H]
    \centering
   \begin{tikzpicture}[scale=.9]
\draw   (3,6)-- (1,4);
\draw   (3,6)-- (3,4);
\draw   (3,6)-- (5,4);
\draw   (3,6)-- (7,4);
\draw   (7,4)-- (7,3);
\draw   (7,4)-- (8,3);
\draw   (7,4)-- (6,3);
\draw   (5,4)-- (4.64,2.97);
\draw   (5,4)-- (5.44,2.99);
\draw (2.6503835914787306,6.8) node[anchor=north west] {$v_0$};
\draw (0.09140054666084313,4.5) node[anchor=north west] {$v_{0,1}$};
\draw (2.066755879502721,4.5) node[anchor=north west] {$v_{0,2}$};
\draw (3.8,4.5) node[anchor=north west] {$v_{0,3}$};
\draw (7.252063628212651,4.5) node[anchor=north west] {$v_{0,4}$};
\begin{scriptsize}
\draw [fill=black] (3,6) circle (1.5pt);
\draw [fill=black] (1,4) circle (1.5pt);
\draw [fill=black] (3,4) circle (1.5pt);
\draw [fill=black] (5,4) circle (1.5pt);
\draw [fill=black] (7,4) circle (1.5pt);
\draw [fill=black] (7,3) circle (1.5pt);
\draw [fill=black] (8,3) circle (1.5pt);
\draw [fill=black] (6,3) circle (1.5pt);
\draw [fill=black] (4.64,2.97) circle (1.5pt);
\draw [fill=black] (5.44,2.99) circle (1.5pt);
\end{scriptsize}
\end{tikzpicture}
    \caption{Tree with a strong support vertex}
    \label{strongsupportvertex}
\end{figure}
Actually, in this case, Albertson index is: 
\[
\irr(T)=\sum_{i=1}^{4}\lvert \deg (v_0)-\deg (v_{0,i})\rvert+2 \times\lvert\deg(v_{0,3})-1\rvert+3\times \lvert\deg(v_{0,4})-1\rvert.
\]

 \begin{hypothesize}~\label{hyfour}
 Let $T$ be tree  of order $n=4$, a degree sequence $\mathscr{D}=(d_1,d_2,d_3,d_4)$ where $d_1 \ge d_2 \ge d_3 \ge d_4$. Then we have: 
  \[
 \irr(T)=\sum_{i=1}^{3}(d_i-1)^2+\sum_{i=1}^{3}(d_4-d_i)+(d_4-1)(d_4-3).
 \]
 \end{hypothesize}
\begin{proof}
Let be consider $\mathscr{D}=(d_1,d_2,d_3,d_4)$ be a degree sequence  with the term $d_4>d_1\geq d_2\geq d_3$, so that according to Proposition~\ref{three.c} for $d_1\geq d_2\geq d_3$ we have $\sum_{i=1}^{3}(d_i-1)^2$ it is known the sum of squared deviations of three values of vertices $d_1,d_2,d_3$ where $d_1\geq d_2\geq d_3$. Thus, furthermore according to Hypothesize~\ref{hyfourtr}, we have
$\sum_{i=1}^{3}(d_i-1)^2=\sum_{i=1}^{3}d_i^2-2\sum_{i=1}^{3}d_i+3$.\par 
According to our term $d_4>d_1\geq d_2\geq d_3$ we notice for vertices $d_4$ it is define $d_4^2-4d_4+3=(d_4-1)(d_4-3)$ true occurs only when $d_4\geq 3$ as we know topological indices quantifying structural irregularity define the sum of degree is $2(n-1)$, So that when $d_4>d_1\geq d_2\geq d_3$ then we have $d_4-d_1$ and $d_4-d_2$ and $d_4-d_3$ we can express that by $\sum_{i=1}^{3}(d_4-d_i)$.

As desire.
\end{proof}
\begin{hypothesize}~\label{hyfourtr}
 Let $T$ be tree  of order $n=4$, a degree sequence $\mathscr{D}=(d_1,d_2,d_3,d_4)$ where $d_1 \ge d_2 \ge d_3 \ge d_4$. Then we have: 
\[
\irr(T)=\begin{cases}
    & \irr_{\max}(T)=\sum_{i=1}^4\left(d_i-1\right)^2+d_1+d_2-d_3-3 d_4+2 \\
	& \irr_{\min}(T)=\sum_{i=1}^4\left(d_i-1\right)^2+d_1-d_2-d_3-d_4+2.
\end{cases}
\]
 \end{hypothesize}
 \begin{corollary}~\label{cor.1}
According to Hypothesize~\ref{hyfourtr} we have the difference between $\irr_{\max}(T),\irr_{\min}(T)$ is define the bound: 
\[
d(\irr_{\max}(T),\irr_{\min}(T))<2d_1.
\]
\end{corollary}
\begin{corollary}~\label{cor.2}
Let be a degree sequence $d=(d_1,d_2,d_3,d_4)$ where $d_1 \geq d_2 \geq d_3 \geq d_4$, then we define the bound which it holds for all valid degree sequences under the given constraints as 
\[
\irr(T) \geq \floor*{\frac{d_1^2+d_4^2}{2}}.
\]
\end{corollary}
Actually, both of Corollary~\ref{cor.1},\ref{cor.2} it is directly results according to Hypothesize~\ref{hyfourtr} based on improvement Proposition~\ref{three.c} as we discussed that and depict in many figures.
\begin{example}~\label{ex.9}
Let be a degree sequence $d=(d_1,d_2,d_3,d_4)$ where $d_1\geq d_2\geq d_3\geq d_4$ define as $d_4\in\{3,4\}$, $d_3\in\{5,6\}$, $d_2\in\{9,12\}$ and $d_1\in\{14,16,18\}$. Then, we provide Table~\ref{tabex.9} for a value of degree sequence according to Albertson index for confirmation Corollary~\ref{cor.1},\ref{cor.2} as we show that.
\begin{table}[H]
    \centering
    \begin{tabular}{|c|c|c|c|c|c|c|c|}
\hline
$(d_1,d_2,d_3,d_4)$ & $\irr_{\max}(T)$ & $\irr_{\min}(T)$ & Diff & $(d_1,d_2,d_3,d_4)$ & $\irr_{\max}(T)$ & $\irr_{\min}(T)$ & Diff \\
\hline
(18,12,6,4) & 454 & 438 & 16 & (18,12,6,3) & 452 & 434 & 18 \\
\hline
(18,12,5,4) & 446 & 430 & 16 & (18,12,5,3) & 444 & 426 & 18 \\
\hline
(18,9,6,4) & 394 & 384 & 10 & (18,9,6,3) & 392 & 380 & 12 \\
\hline
(18,9,5,4) & 386 & 376 & 10 & (18,9,5,3) & 384 & 372 & 12 \\
\hline
(16,12,6,4) & 388 & 372 & 16 & (16,12,6,3) & 386 & 368 & 18 \\
\hline
(16,12,5,4) & 380 & 364 & 16 & (16,12,5,3) & 378 & 360 & 18 \\
\hline
(16,9,6,4) & 328 & 318 & 10 & (16,9,6,3) & 326 & 314 & 12 \\
\hline
(16,9,5,4) & 320 & 310 & 10 & (16,9,5,3) & 318 & 306 & 12 \\
\hline
(14,12,6,4) & 330 & 314 & 16 & (14,12,6,3) & 328 & 310 & 18 \\
\hline
(14,12,5,4) & 322 & 306 & 16 & (14,12,5,3) & 320 & 302 & 18 \\
\hline
(14,9,6,4) & 270 & 260 & 10 & (14,9,6,3) & 268 & 256 & 12 \\
\hline
(14,9,5,4) & 262 & 252 & 10 & (14,9,5,3) & 260 & 248 & 12 \\
\hline
\end{tabular}
\caption{Degree Sequence according to the term $d_1\geq d_2\geq d_3\geq d_4$.}\label{tabex.9}
\end{table}
Then we have: 
\[
\begin{cases}
    \irr_{\max}(T)=454 \quad \text{ when } d=(18,12,6,4), \\
  \irr_{\min}(T)=248 \quad \text{ when } d=(14,9,5,3).
\end{cases}
\]
\end{example}
\begin{proposition}~\label{caterpillarvertex}
Let $\mathcal{C}(n,m)$ be a caterpillar tree with $n$ vertices and $m$ pendent vertices, if $\mathcal{C}(n,m)$ has a maximum Albertson index, then $\mathcal{C}(n,m)$ has a strong support vertex.
\end{proposition}
\begin{proof}
Let $T=(V,E)$ be a tree of order $n$ with vertices set $V=\{v_1,\dots,v_i\}$ and edges set $E=\{e_1,\dots,e_j\}$. The degrees of the vertices are given as follows: $\deg(v_1)=3, \deg(v_2)=5,\deg(v_3)=7, \dots,\deg(v_i)=p$, where $p$ is prime number. The sum of the degrees of all vertices in the tree must be 
\begin{equation}~\label{eqsr01}
    \sum_{v\in V}\deg(v)=2(n-1),
\end{equation}
when $p$ is prime number, then~(\ref{eqsr01}) should be 
\begin{equation}~\label{eqsr02}
    \sum_{k=1}^{m}(2k+1)=m(m+2),
\end{equation}
Typically, vertices set $V=\{v_1,\dots,v_i\}$ not specified with high degrees are often leaves (degree 1), especially when constructing a tree with given degrees. Thus, let assume remains $n-i$ vertices are leaves with degree 1, then we have
\begin{equation}~\label{eqsr03}
 \sum_{v\in V}\deg(v)= \sum_{k=1}^{n}(2k+1)+(n-m)=m(m+2)+(n-m).
\end{equation}
By comparing~(\ref{eqsr01}),(\ref{eqsr03}) we have $n=m^2+m+2$. Since there are many leaves $m^2 + 2$ and only $m$ non-leaf vertices, each $v_i$ is likely to be adjacent to at least one leaf. In a star-like structure or a path among $v_1, v_2, \dots, v_i$, the remaining degree of each $v_i$ is used to connect to leaves.
\end{proof}
\begin{proposition}~\label{classoftrees}
Let $\mathcal{T}_{n, \Delta}$ be a class of trees with $n$ vertices, let $T \in \mathcal{T}_{n, \Delta}$ be a tree, and let $v_0 \in V(T)$ be a vertex with maximum degree $\Delta$. For any support vertex $v_{\ell}$ in $T$, different from $v_0$, where $\deg(v_{\ell})\geqslant 3$, then there exists another tree $T^{\prime} \in \mathcal{T}_{n, \Delta}$ such that $\irr(T^{\prime}) < \irr(T)$.
\end{proposition}
\begin{proof}
Let $T$ be a tree with $x$ vertex, let $y\neq x$ be a strong support vertex where $\deg(y)=\lambda\geqslant 3$ where vertex $y$ has maximum degree $\Delta$ in $T$, let $\mathscr{N}_T(y)=\{y_1,y_2,\dots,y_{\lambda}\}$ be an open neighborhood  of $y$. In this case, we say $y_{\lambda}$ might be in $y$ or not, let $T^{\prime}$ be a tree compute from $T-\{y_1,y_2\}$ by linking with a pendant edge $y_2y_1$, where $T^{\prime} \in \mathcal{T}_{n, \Delta}$, then according to Definition~\ref{Strongsupportvertex}, noticed that $\deg_{T^{\prime}}(y) = \deg_T(y) - 1$, $\deg_{T^{\prime}}(y_1) = \deg_T(y_1)$ and $\deg_{T^{\prime}}(y_2) = d_T(y_2) + 1$, thus $\deg_{T^{\prime}}(y_i) = \deg_T(y_i)$ for $i = 3, \ldots, \lambda$, then
\begin{align*}
   \irr(T) -\irr(T^{\prime})&=\sum_{uv\in E(T)}\lvert \deg_{T}(u)-\deg_{T}(v)\rvert-\sum_{uv\in E(T^{\prime})}\lvert \deg_{T^{\prime}}(u)-\deg_{T^{\prime}}(v)\rvert\\
   &=\lvert \deg_{T}(y_1)-\deg_{T}(y)\rvert+\lvert \deg_{T}(y_2)-\deg_{T}(y)\rvert+\sum_{i=3}^{\lambda}\lvert \deg_{T}(y)-\deg_{T}(y_i)\rvert-\\
   &-\lvert \deg_{T^{\prime}}(y_1)-\deg_{T^{\prime}}(y_2)\rvert-\lvert \deg_{T^{\prime}}(y_2)-\deg_{T^{\prime}}(y)\rvert-\sum_{i=3}^{\lambda} \lvert (\deg_{T}(y)-1)-\deg_{T}(y_i)\rvert\\
   &=2\lambda-3+\sum_{i=3}^{\lambda}\lvert \lambda-\deg_{T}(y_i)\rvert +\lambda-3-\sum_{i=3}^{\lambda} \lvert (\lambda-1)-\deg_{T}(y_i)\rvert \\
&<3\lambda-6.
\end{align*}
Thus, we have the inequality
\begin{equation}~\label{eqcla021}
    3\lambda-6>0.
\end{equation}
holds $\irr(T) -\irr(T^{\prime})>0$.
As desire.
\end{proof}
\begin{proposition}~\label{pro.se.1}
Let $\mathcal{T}_{n, \Delta}$ be a class of trees with $n$ vertices, let $\mathcal{T}_1,\mathcal{T}_2$ be a trees of order $n$, let $\mathscr{D}_i=(x_1,x_2,\dots, x_{n-1})$ and $\mathscr{D}_j=(y_1,y_2,\dots, y_{n-1})$ be tow non increasing degree sequence, where $\sum_{i=0}^{n-1}d_i\leqslant \sum_{j=0}^{n-1}d_j$, then we have: 
\[
\irr(\mathcal{T}_{\mathscr{D}_i})\leqslant \irr(\mathcal{T}_{\mathscr{D}_j}).
\]
\end{proposition}
\begin{proof}
Assume a degree sequence $\mathscr{D}_i=(x_1,x_2,\dots, x_{n-1})$ where $x_1\geqslant x_2\geqslant \dots\geqslant x_{n-1}$, then we have: 
\begin{align*}
\irr(\mathcal{T}_{\mathscr{D}_i}) &= \lvert \deg x_1-\deg x_2\rvert +\lvert \deg x_2-\deg x_3\rvert+\dots+\lvert \deg x_{n-2}-\deg x_{n-1}\rvert+\\
&+x_{1}^{2}+x_{n-1}^{2}+\sum_{i=2}^{n-2}\lvert \deg x_i+2\rvert\cdot\lvert \deg x_i-1\rvert-2.\\
&=\sum_{i=1}^{n-1} \lvert \deg x_i-\deg x_{i+1}\rvert+x_{1}^{2}+x_{n-1}^{2}+\sum_{i=2}^{n-2}\lvert \deg x_i+2\rvert\cdot\lvert \deg x_i-1\rvert-2.
\end{align*}
For a degree sequence $\mathscr{D}_j=(y_1,y_2,\dots, y_{n-1})$ where $y_1\geqslant y_2\geqslant \dots\geqslant y_{n-1}$, we have: 
\[
\irr(\mathcal{T}_{\mathscr{D}_j})=\sum_{j=1}^{k-1} \lvert \deg y_j-\deg y_{j+1}\rvert+y_{1}^{2}+y_{k-1}^{2}+\sum_{j=2}^{k-2}\lvert \deg y_j+2\rvert\cdot\lvert \deg y_j-1\rvert-2.
\]
According to the term $\sum_{i=0}^{n-1}d_i\leqslant \sum_{j=0}^{n-1}d_j$ we have two cases as: 
\case{1} If $\sum_{i=0}^{n-1}d_i= \sum_{j=0}^{n-1}d_j$, then for a degree sequence in non-increasing order we have $\sum_{i=0}^{n-1}d_i=2m$ and $\sum_{j=0}^{n-1}d_j=2m$ where  $m$ is the number of edges, then $(x_1,x_2,\dots, x_{n-1})=(y_1,y_2,\dots, y_{n-1})$ if and only if $\mathcal{T}_1$ is isomorphism with $\mathcal{T}_2$ denote by: $\mathcal{T}_1 \cong \mathcal{T}_2$.
\case{2} If $\sum_{i=0}^{n-1}d_i< \sum_{j=0}^{n-1}d_j$, then for a degree sequence in non-increasing order we have vertex sets $V(\mathcal{T}_1) = \{v_1, v_2, \dots, v_{n-1}\}$  and $V(\mathcal{T}_2) = \{w_1, w_2, \dots, w_{n-1}\}$, if $\mathcal{T}_1 \cong \mathcal{T}_2$ then we define isomorphism $f:\mathcal{T}_1 \to \mathcal{T}_2$ where $f(v_i)=w_{\eta(i)}$, then we have: $\deg(v_i) = \deg(w_{\eta(i)})$, in this case, we have $x_i = y_i$, then $(x_1,x_2,\dots, x_{n-1})=(y_1,y_2,\dots, y_{n-1})$.\par
\noindent
Therefore, in both cases, if both a degree sequence $\mathscr{D}_i$ and 
 a degree sequence $\mathscr{D}_j$ satisfy $\sum_{i=0}^{n-1}d_i\leqslant \sum_{j=0}^{n-1}d_j$ when $\mathcal{T}_1 \cong \mathcal{T}_2$, then $\irr(\mathcal{T}_{\mathscr{D}_i})\leqslant \irr(\mathcal{T}_{\mathscr{D}_j})$. 

 As desire.
\end{proof}
\begin{lemma}~\label{albisomorphic}
Let $T_1$ and $T_2$ be a star trees of order $n$, where $T_1 \cong T_2$, then Albertson index of $T_1 \cong T_2$ are: 
\[
\irr(T_1 \cong T_2)=2\times\sum_{i=2}^{n+1}\lvert \deg(u)-\deg(v)\rvert.
\]
\end{lemma}
\begin{proof}
Let $T_1$ and $T_2$ be a trees given in Figure~\ref{fig:Isomorphic}, a star graph with $n$ leaves. Then, for tree $T_1$ we have a tree with $n+1$ vertices, thus we noticed that $\deg(v_1)=n$ and $\deg(v_2)=\deg(v_3)=\dots=\deg(v_n)=1$, then 
\begin{equation}~\label{eqison1}
\irr(T_1)=\sum_{i=2}^{n+1}\lvert \deg_{T_1}(u)-\deg_{T_1}(v)\rvert=n(n-1).
\end{equation}
For tree $T_2$ we have a tree with $n+1$ vertices, thus we noticed that $\deg(u_1)=n$ and $\deg(u_2)=\deg(u_3)=\dots=\deg(u_n)=1$, then as in~(\ref{eqison1}) we have: 
\begin{equation}~\label{eqison2}
\irr(T_2)=\sum_{i=2}^{n+1}\lvert \deg_{T_2}(u)-\deg_{T_2}(v)\rvert=n(n-1).
\end{equation}
From~(\ref{eqison1}),(\ref{eqison2}) and according to Definition~\ref{TreeIsomorphic}, we compute the Albertson index for clarity, then we have: 
\[
\irr(T_1 \cong T_2)=\irr(T_1)+\irr(T_2).
\]
As desire.
\end{proof}

\begin{theorem}~\label{resn1}
Let $T=(V,E)$ be a tree of order $n$, let $\mathscr{D}_1=(d_1,d_2,\dots,d_i)$, $\mathscr{D}_2=(c_1,c_2,\dots,c_j)$ be a tow degree sequences, where $i\neq j$, if there a vertex $v\in V(T)$ with $\deg(v)>1$, then satisfying
\[
\irr(T)+\deg(v)+\sum_{k=1}\deg(d_k)-2\geqslant \irr(T) +\deg(v)+\sum_{k=1}\deg(c_k).
\]
\end{theorem}
\begin{proof}
Assume $\deg(v)\geqslant \min(c_j)$ where $c_j\in \mathscr{D}_2$, $j\in \mathbb{N}$, then  according to Proposition~\ref{pro.se.1} we provide $\irr(\mathcal{T}_{\mathscr{D}_i})\leqslant \irr(\mathcal{T}_{\mathscr{D}_j})$. Now, we have 
$\irr_{\max}(T)-\irr_{\min}(T)>1$ when 
\[
\sum_{k=1}\deg(d_k)-\sum_{k=1}\deg(c_k)>\deg(v) \quad \forall \hspace{0.1cm} v\in V(T), 
\]
then we have: 
\begin{align*}
  \deg_{\max}(v)-\deg_{\min}(v)&=\sum_{j=1} (\deg(v_{j+1},v_j)-\deg(v_j,v_{j-1}))\\
  &\leqslant \sum_{j=1} \deg_{\max}(v_j)-\deg_{\min}(v_j)
\end{align*}
for any tree we have the constant term is $-2$ thus,  represents adding the constant 
\[
\deg_{\max}(v)-\deg_{\min}(v)-2\leqslant  \sum_{j=1} \deg_{\max}(v_j)-\deg_{\min}(v_j)-2n.
\]
Therefore, according to the term $i\neq j$ we have tow cases as we express that.
\case{1} In this case we consider $i>j$, then, to simplify the notation  we have
\begin{itemize}
    \item When $\sum_{k=1}\deg(d_k)>\sum_{k=1}\deg(c_k)$ then, the inequality become
\begin{equation}~\label{eqjan1}
    \sum_{k=1}\deg(d_k)-2\geqslant \sum_{k=1}\deg(c_k).
\end{equation}
    \item When $\sum_{k=1}\deg(d_k)<\sum_{k=1}\deg(c_k)$ then, we have
    \[
    \sum_{k=1}\deg(d_k)-2\leqslant \sum_{k=1}\deg(c_k).
    \]
    But we considered $i>j$ thus the inequality holds to
    \[
    \sum_{k=1}\deg(d_k)-2= \sum_{k=1}\deg(c_k).
    \]
\end{itemize}
\case{2} In this case we consider $i<j$, then, to simplify the notation  we have
\begin{itemize}
    \item When $\sum_{k=1}\deg(d_k)<\sum_{k=1}\deg(c_k)$ then, the inequality by using Proposition~\ref{pro.se.1} become for Albertson index among this degree sequence $\mathscr{D}_1$,$\mathscr{D}_2$ as we show that
  \begin{equation}~\label{eqjan2}
      \sum_{k=1}\deg(d_k)-2\leqslant \sum_{k=1}\deg(c_k),
  \end{equation}
   the sum $\sum_{i=1}^{n} (-2)$ represents adding the constant $-2(n)$ times. Thus, according to the constant term in degree sequence $\mathscr{D}_2$ we noticed that~(\ref{eqjan2}) holds to~(\ref{eqjan1}). 
    \item When $\sum_{k=1}\deg(d_k)>\sum_{k=1}\deg(c_k)$ this case holds immediately to~(\ref{eqjan1}) with considered $i<j$.
\end{itemize}
The inequalities for the degree sequences $\mathscr{D}_1$ and $\mathscr{D}_2$  under cases $i>j$  and $i<j$  show that 
\[
\sum_{k=1}\deg(d_k) - 2 \geqslant \sum_{k=1}\deg(c_k)
\]
 depending on the relative magnitudes. Specifically, when 
$\sum_{k=1}\deg(d_k) < \sum_{k=1}\deg(c_k)$ the inequality $\sum_{k=1}\deg(d_k) - 2 \leqslant \sum_{k=1}\deg(c_k)$ aligns with the previous case. In case 2 reinforces the consistency of these relations across degree sequences.
\end{proof}

\section{Sigma Index Among Strong Support Vertex}~\label{sec4}
In this section, let $\mathcal{T}_{n, \Delta}$ be a class of trees with $n$ vertices, in Proposition~\ref{classoftreessigma}, we presented for any support vertex $v_{\ell}$ in $T$, different from $v_0$, where $3<\deg(v_{\ell})<10$. Also, in Proposition~\ref{classoftreessigman2}, we presented that with term $\deg(v_{\ell})\geqslant 11$.
\begin{hypothesize}~\label{hy.sigma5}
Let $T$ be a tree of order $n>0$, and let $\mathscr{D}=(d_1,\dots,d_5)$  be a degree sequence where $d_5\geqslant d_4 \geqslant d_3 \geqslant d_2 \geq d_1$, then Sigma index of $T$ given by: 
\[
\sigma(T)=\sum_{i=1}^{3}d_i(d_{i+1})^2+(d_1-1)^3+(d_4)^3 +\sum_{i=1}^{4}(d_i-d_{i+1})^2.
\]
where inequality holds if and only if $d_i=d_{i-1}+1$.
\end{hypothesize}
\begin{proposition}~\label{classoftreessigma}
Let $\mathcal{T}_{n, \Delta}$ be a class of trees with $n$ vertices, let $T \in \mathcal{T}_{n, \Delta}$ be a tree, and let $v_0 \in V(T)$ be a vertex with maximum degree $\Delta$. For any support vertex $v_{\ell}$ in $T$, different from $v_0$, where $3<\deg(v_{\ell})<10$, then there exists another tree $T^{\prime} \in \mathcal{T}_{n, \Delta}$  that $\sigma(T^{\prime}) < \sigma(T)$.
\end{proposition}
\begin{proof}
Let $T$ be a tree with $x$ vertex, let $y\neq x$ be a strong support vertex where $\deg(y)=\lambda\geqslant 3$ where vertex $y$ has maximum degree $\Delta$ in $T$, let $\mathscr{N}_T(y)=\{y_1,y_2,\dots,y_{\lambda}\}$ be an open neighborhood  of $y$. In this case, we say $y_{\lambda}$ might be in $y$ or not, let $T^{\prime}$ be a tree compute from $T-\{y_1,y_2\}$ by linking with a pendant edge $y_2y_1$, where $T^{\prime} \in \mathcal{T}_{n, \Delta}$, then according to Definition~\ref{Strongsupportvertex} and Proposition~\ref{classoftrees}, noticed that $\deg_{T^{\prime}}(y) = \deg_T(y) - 1$, $\deg_{T^{\prime}}(y_1) = \deg_T(y_1)$ and $\deg_{T^{\prime}}(y_2) = d_T(y_2) + 1$, thus $\deg_{T^{\prime}}(y_i) = \deg_T(y_i)$ for $i = 3, \ldots, \lambda$, then
\begin{align*}
   \sigma(T) -\sigma(T^{\prime})&=\sum_{uv\in E(T)}\left( \deg_{T}(u)-\deg_{T}(v)\right)^2-\sum_{uv\in E(T^{\prime})}\left( \deg_{T^{\prime}}(u)-\deg_{T^{\prime}}(v)\right)^2\\
   &=\left( \deg_{T}(y_1)-\deg_{T}(y)\right)^2+\left( \deg_{T}(y_2)-\deg_{T}(y)\right)^2+\sum_{i=3}^{\lambda}\left( \deg_{T}(y)-\deg_{T}(y_i)\right)^2-\\
   &-\left( \deg_{T^{\prime}}(y_1)-\deg_{T^{\prime}}(y_2)\right)^2-\left( \deg_{T^{\prime}}(y_2)-\deg_{T^{\prime}}(y)\right)^2-\sum_{i=3}^{\lambda} \left( (\deg_{T}(y)-1)-\deg_{T}(y_i)\right)^2\\
   &=2\lambda^2-6\lambda+5+\sum_{i=3}^{\lambda}\left( \lambda-\deg_{T}(y_i)\right)^2-(\lambda+2)^2-1-\sum_{i=3}^{\lambda} \left( (\lambda-1)-\deg_{T}(y_i)\right)^2\\
   &<2\lambda^2-6\lambda+5-(\lambda+2)^2-1\\
   &=\lambda^2-10\lambda<0\quad \text{ when } 3<\lambda <10
\end{align*}
As desire.
\end{proof}
\begin{proposition}~\label{classoftreessigman2}
Let $\mathcal{T}_{n, \Delta}$ be a class of trees with $n$ vertices, let $T \in \mathcal{T}_{n, \Delta}$ be a tree, and let $v_0 \in V(T)$ be a vertex with maximum degree $\Delta$. For any support vertex $v_{\ell}$ in $T$, different from $v_0$, where $\deg(v_{\ell})\geqslant 11$, then there exists another tree $T^{\prime} \in \mathcal{T}_{n, \Delta}$  that $\sigma(T^{\prime}) > \sigma(T)$.
\end{proposition}
\begin{proof}
Immediately from proof of Proposition~\ref{classoftreessigma}.
\end{proof}
\begin{corollary}~\label{cor.3}
Let  $T$ be a tree of order $n$, then satisfying: 
\begin{enumerate}
    \item For $d_3>3$, then
    \[
\log_{d_4-2}\left(\frac{2d_4-4}{d_3-1}\right) < 2 + \left\lfloor \frac{d_4-2}{d_3-1} \right\rfloor.
\]
\item For $n>2$, then we have: 
\[
\log_{2n-1}\left( 2d_i\right)_{i>1}+\sqrt{\left( 2d_i\right)_{i>1}^2-1}\geqslant 2+\lfloor\left( 2d_i\right)_{i>1}-1\rfloor.
\]
\end{enumerate}
\end{corollary}
\begin{theorem}~\label{thm.sigma}
Let $T$ be a tree of order $n>0$, define degree sequence $\mathscr{D}=(d_1,\dots,d_n)$ where $d_n\geqslant \dots \geqslant d_1$, then Sigma index among tree given as:
\[
\sigma(T)=\sum_{i\in\{1,n\}}(d_i+1)(d_i-1)^2+\sum_{i=2}^{n-1}(d_i+2)(d_i-1)^2+\sum_{i=2}^{n-1}(d_i-d_{i+1})^2+2n-2.
\]
\end{theorem}
\begin{proof}
It is well known Sigma index as $\sigma(G)=\sum_{uv\in E(G)}(d_u-d_v)^2$, let  $\mathscr{D}=(d_1,d_2,\dots,d_n)$ be a degree sequence of order $n\geqslant1$ where $d_n\geqslant d_{n-1} \geqslant \dots \geqslant d_2 \geqslant d_1$, then we have following cases.
\case{1} Sigma index among degree sequence given as $\sigma(T)=\sum_{i=1}^{n-1}(d_i+1)(d_i-1)^2+\sum_{i=1}^{3}(d_i-d_{i+1})^2-2$. Also, we found $\sigma(T)=\sum_{i=1}^{2} (d_i+1)(d_i-1)^2+\sum_{i=1}^{4}(d_i-d_{i+1})^2+\sum_{i=2}^{3}(d_i+2)(d_i-1)^2-2$, by considering the maximum and minimum values in both cases and the terms constant. Actually, for a degree sequence of order 6, we have $\sigma(T)=\sum_{i\in\{1,n\}}(d_i+1)(d_i-1)^2+\sum_{i=2}^{n-1}(d_i-d_{i+1})^2+\sum_{i=2}^{n}(d_i+2)(d_i-1)^2+10$.
So that, the relationship is correct for $n\in[3,6]$, then we need to prove that for $n$, then we prove that for $n+1$.\par
let be assume the sequence $\mathscr{D}=(d_1,\dots,d_n)$ where $d_n\geqslant \dots \geqslant d_1$, then we have:
\case{2} in this case, we have:
\begin{align*}
\sigma(T)&=(d_1+1)(d_1-1)^2+(d_n+1)(d_n-1)^2+(d_1-d_2)^2+\dots+(d_{n-1}-d_n)^2+\\
&+(d_2+2)(d_2-1)^2+(d_3+2)(d_3-1)^2+\dots+(d_{n-1}+2)(d_{n-1}-1)^2\\
&=\sum_{i\in\{1,n\}}(d_i+1)(d_i-1)^2+\sum_{i=2}^{n-1}(d_i+2)(d_i-1)^2+\sum_{i=1}^{n}(d_i-d_{i+1})^2.
\end{align*}
\case{3} In this case, let be shifting the sequence as $d=(d_{n-1}, \dots, d_2)$, then we have: 
\begin{align*}
\sigma(T)&=(d_2+1)(d_2-1)^2+(d_{n-1}+1)(d_{n-1}-1)^2+(d_1-d_{n-2})^2+\dots+(d_n-d_2)^2+\\
&+(d_1+2)(d_1-1)^2+(d_{n-2}+2)(d_{n-2}-1)^2+\dots+(d_n+2)(d_n-1)^2\\
&=\sum_{i\in\{2,n-1\}}(d_i+1)(d_i-1)^2+\sum_{i=2}^{n-1}(d_i+2)(d_i-1)^2+\sum_{i=1}^{n}(d_i-d_{i+1})^2.
\end{align*}

So, that the relationship is correct for $n$, now let be prove that for $n+1$, in this case we have the sequence is $\mathscr{D}=(d_2,\dots, d_{n+1})$ where $d_{n+1}\geqslant \dots \geqslant d_1$. 
\case{4} In this case, we have: 
\begin{align*}
\sigma(T)&=(d_2+1)(d_2-1)^2+(d_{n+1}+1)(d_{n+1}-1)^2+(d_2-d_3)^2+\dots+(d_{n}-d_{n+1})^2+\\
&+(d_3+2)(d_3-1)^2+(d_4+2)(d_4-1)^2+\dots+(d_n+2)(d_n-1)^2\\
&=\sum_{i\in\{2,n+1\}}(d_i+1)(d_i-1)^2+\sum_{i=3}^{n}(d_i+2)(d_i-1)^2+\sum_{i=2}^{n+1}(d_i-d_{i+1})^2.
\end{align*}

\case{5} In this case, let be shifting the sequence as $\mathscr{D}=(d_n, \dots, d_3)$, then we have:
\begin{align*}
\sigma(T)&=(d_3+1)(d_3-1)^2+(d_n+1)(d_n-1)^2+(d_3-d_{n-1})^2+\dots+(d_{n+1}-d_2)^2+\\
&+(d_2+2)(d_2-1)^2+(d_{n-1}+2)(d_{n-1}-1)^2+\dots+(d_{n+1}+2)(d_{n+1}-1)^2\\
&=\sum_{i\in\{3,n\}}(d_i+1)(d_i-1)^2+\sum_{i=2}^{n}(d_i+2)(d_i-1)^2+\sum_{i=1}^{n+1}(d_i-d_{i+1})^2.
\end{align*}
So that the relationship is correct for $n+1$.

As desire.
\end{proof}
\begin{example}
Let be a degree sequence is $d=(4,8,10,14,18,20)$, then for Max sigma we have: 
\begin{align*}
&[14802, 18, 4, 10, 14, 8, 20], [14802, 18, 4, 10, 14, 8, 20], [14802, 18, 4, 14, 8, 10, 20], [14802, 18, 4, 14, 10, 8, 20], \\
& [14802, 18, 8, 10, 14, 4, 20], [14802, 18, 8, 14, 4, 10, 20], [14802, 18, 8, 14, 10, 4, 20], [14802, 18, 10, 4, 14, 8, 20], \\
& [14802, 18, 10, 8, 14, 4, 20], [14802, 20, 4, 10, 14, 8, 18], [14802, 20, 4, 14, 8, 10, 18], [14802, 20, 4, 14, 10, 8, 18], \\
& [14802, 20, 8, 10, 14, 4, 18], [14802, 20, 8, 14, 4, 10, 18], [14802, 20, 8, 14, 10, 4, 18], [14802, 20, 10, 4, 14, 8, 18], \\
&[14802, 20, 10, 8, 14, 4, 18].
\end{align*}
and for Min sigma we have: 
\begin{align*}
&[14196, 4, 10, 14, 18, 20, 8], [14196, 4, 10, 14, 18, 20, 8], [14196, 4, 10, 14, 20, 18, 8], [14196, 4, 10, 18, 20, 14, 8],\\
& [14196, 4, 10, 20, 18, 14, 8], [14196, 4, 14, 18, 20, 10, 8], [14196, 4, 14, 20, 18, 10, 8], [14196, 4, 18, 20, 14, 10, 8], \\
& [14196, 4, 20, 18, 14, 10, 8], [14196, 8, 10, 14, 18, 20, 4], [14196, 8, 10, 14, 20, 18, 4], [14196, 8, 10, 18, 20, 14, 4], \\
& [14196, 8, 10, 20, 18, 14, 4], [14196, 8, 14, 18, 20, 10, 4], [14196, 8, 14, 20, 18, 10, 4], [14196, 8, 18, 20, 14, 10, 4],\\
&[14196, 8, 20, 18, 14, 10, 4].
\end{align*}

\end{example}
\section{Conclusion}\label{sec5}
Through this paper, we presented both of topological index (Albertson and Sigma) among strong support vertex, this results implied clearly in Proposition~\ref{caterpillarvertex}, and in Proposition~\ref{classoftrees} as $\irr(T^{\prime}) < \irr(T)$. \par 
Sigma index among tree with strong support vertex given in Proposition~\ref{classoftreessigman2} with term $\deg(v_{\ell})\geqslant 11$, later,in Theorem~\ref{thm.sigma}, we provide Sigma index among trees as:
\[
\sigma(T)=\sum_{i\in\{1,n\}}(d_i+1)(d_i-1)^2+\sum_{i=2}^{n-1}(d_i+2)(d_i-1)^2+\sum_{i=2}^{n-1}(d_i-d_{i+1})^2+2n-2.
\]
\section*{Acknowledgments}
I would like to extend my deepest gratitude to Prof.Belov Alexey Yakovlevich.

\end{document}